\documentclass[12pt,english]{amsart}
\usepackage[latin9]{inputenc}
\usepackage{babel}
\usepackage{amstext}
\usepackage{amsthm}
\usepackage{amssymb}
\usepackage[unicode=true,pdfusetitle,
 bookmarks=true,bookmarksnumbered=false,bookmarksopen=false,
 breaklinks=false,pdfborder={0 0 1},backref=false,colorlinks=false]
 {hyperref}

\makeatletter
\numberwithin{equation}{section}
\numberwithin{figure}{section}
\theoremstyle{plain}
\newtheorem{thm}{\protect\theoremname}[section]
\theoremstyle{plain}
\newtheorem{cor}[thm]{\protect\corollaryname}
\theoremstyle{plain}
\newtheorem{lem}[thm]{\protect\lemmaname}
\theoremstyle{plain}
\newtheorem{prop}[thm]{\protect\propositionname}
\theoremstyle{remark}
\newtheorem{rem}[thm]{\protect\remarkname}

\usepackage{geometry}
\makeatother

\providecommand{\corollaryname}{Corollary}
\providecommand{\lemmaname}{Lemma}
\providecommand{\propositionname}{Proposition}
\providecommand{\remarkname}{Remark}
\providecommand{\theoremname}{Theorem}

\begin{document}
\title{Densities of eigenspaces of frobenius and Distributions of R-Modules}
\author{Jack Klys, Jacob Tsimerman}
\begin{abstract}
For any polynomial $p\left(x\right)$ over $\mathbb{F}_{l}$ we determine
the asymptotic density of hyperelliptic curves over $\mathbb{F}_{q}$
of genus $g$ for which $p\left(x\right)$ divides the characteristic
polynomial of Frobenius acting on the $l$-torsion of the Jacobian,
and give an explicit formula for this density. We prove this result
as a consequence of more general density theorems for quotients of
Tate modules of such curves, viewed as modules over the Frobenius.
The proof involves the study of measures on $R$-modules over arbitrary
rings $R$ which are finite $\mathbb{Z}_{l}$-algebras. In particular
we prove a result on the convergence of sequences of such measures,
which can be applied to the moments computed in \cite{lipnowskitsimerman}
to obtain the above results. We also extend the random model for finite
$R$-modules proposed in \cite{lipnowskitsimerman} to such rings
$R$, and prove several of its properties. Notably the measure obtained
is in general not inversely proportional to the size of the automorphism
group.

\end{abstract}

\maketitle

\section{Introduction}

\subsection{Background}

Cohen, Lenstra and Martinet \cite{cohenlenstra,cohenmartinet} put
forth heuristics for predicting the distribution of the $l$-part
of the class group of certain families of number fields. Their proposed
distribution on the set of finite abelian $l$-groups in the most
simple case is 
\[
\mu_{CL}\left(A\right)=\frac{c}{\left|\mathrm{Aut}\left(A\right)\right|}
\]
 for some normalizing constant $c$. This distribution can be modelled
by cokernels of Haar-random $n\times n$ matrices over $\mathbb{Z}_{l}$,
in the limit $n\longrightarrow\infty$ \cite{washingtonRandomMatrix}.
We will assume $l$ is an odd prime throughout the paper, and $q$
is a power of a prime distinct from $l$.

These heuristics have been extended to the setting of Jacobians of
curves over finite fields. Let $\mathcal{M}_{g}\left(\mathbb{F}_{q}\right)$
be the space of smooth projective curves over $\mathbb{F}_{q}$. In
their simplest form the heuristics say that the distribution of the
finite abelian group $\mathrm{Jac}\left(C\right)\left(\mathbb{F}_{q}\right)\left[l^{\infty}\right]$
in the $g$-limit is the Cohen-Lenstra distribution $\mu_{CL}$. As
a consequence of ground-breaking results on homological stabilization
Ellenberg, Venkatesh, and Westerland \cite{EVW} were able to prove
this heuristic holds in the $g,q$-limit.

Lipnowski and the second author \cite{lipnowskitsimerman} further
extended the above heuristics, essentially by considering the whole
group $\mathrm{Jac}\left(C\right)\left[l^{\infty}\right]$ as a module
over the Frobenius operator $F$. Though there is no distribution
on $\mathrm{Jac}\left(C\right)\left[l^{\infty}\right]$, this is remedied
by considering the distribution of $\ker P\left(F\right)$ for any
suitable polynomial $P\left(X\right)\in\mathbb{Z}_{l}\left[X\right]$.
Indeed $\mathrm{Jac}\left(C\right)\left(\mathbb{F}_{q}\right)\left[l^{\infty}\right]=\ker\left(1-F\right)$
so this generalizes the previous setting.

Furthermore they developed and studied a similar random model for
finite modules over rings $R$ which are finite local $\mathbb{Z}_{l}$-algebras
containing $\mathbb{Z}_{l}$ by considering the distribution of cokernels
of large Haar-random matrices over $R$. This gives rise to a measure
on finite $R$-modules denoted $\mu_{R}$. They conjectured when $R=\mathbb{Z}_{l}\left[F\right]/\left\langle P\left(F\right)\right\rangle $
this models the distribution of $\mathrm{Jac}\left(C\right)\left[l^{\infty}\right]\left[P\left(F\right)\right]$.
Extending the methods of Ellenberg-Venkatesh-Westerland \cite{EVW}
they proved that as $C$ varies over $\mathcal{M}_{g}\left(\mathbb{F}_{q}\right)$
the moments converge to 1 in the $g,q$ limit, 
\[
\mathbb{E}_{\mathcal{M}_{g}\left(\mathbb{F}_{q}\right)}\left[\mathrm{Surj}_{\mathbb{Z}_{l}\left[F\right]}\left(\mathrm{Jac}\left(C\right),G\right)\right]\longrightarrow1
\]
for all finite $\mathbb{Z}_{l}\left[F\right]$-modules $G$. 

In this paper we determine the density of hyperelliptic curves in
$\mathcal{M}_{g}\left(\mathbb{F}_{q}\right)$ for which any polynomial
$P\left(X\right)\in\mathbb{F}_{l}\left[X\right]$ divides the characteristic
polynomial of Frobenius acting on $\mathrm{Jac}\left(C\right)\left[l\right]$
in the $g,q$-limit, and we give an explicit formula for this density.
We obtain this and other results about the structure of torsion $R$-submodules
of $\mathrm{Jac}\left(C\right)\left[l^{\infty}\right]$ as a consequence
of the more general result that the density of $\mathrm{Jac}\left(C\right)\left[l^{\infty}\right]\left[P\left(F\right)\right]$
as a $\mathbb{Z}_{l}\left[F\right]/\left\langle P\left(F\right)\right\rangle $-module
converges weakly to $\mu_{\mathbb{Z}_{l}\left[F\right]/\left(P\left(F\right)\right)}$
in the $g,q$-limit.

In particular we prove a technical result which says that if $R$
is any finite $\mathbb{Z}_{l}$-algebra, then for any sequence of
measures on the set of $R$-modules convergence of moments implies
convergence of measures under suitable conditions. We combined this
with the above mentioned moment results of \cite{lipnowskitsimerman}.

To obtain explicit formulas in our results we also extend the random
model in \cite{lipnowskitsimerman} to any finite $\mathbb{Z}_{l}$-algebra
$R$ and prove all the analogous properties in this case. The measure
is no longer inversely proportional to the size of the automorphism
group in general, though this does hold if for example $R$ contains
$\mathbb{Z}_{l}$. We also prove an explicit formula for the distribution
of ranks and its moments in this model.

\subsection{Jacobians of hyperelliptic curves}

Our main result is the following explicit formula for the densities
of factors of the characteristic polynomial of Frobenius acting on
$\mathrm{Jac}\left(C\right)$$\left[l\right]$.

For any ring $R$ let $S_{R}$ be the set of finite $R$-modules.
We will use $\mathrm{Prob}\left(\cdot\right)$ to denote the density
of points in $\mathcal{M}_{g}\left(\mathbb{F}_{q}\right)$ for which
some property holds (implicitly depending on $q$ and $g$). For $C\in\mathcal{M}_{g}\left(\mathbb{F}_{q}\right)$
let $P_{C}\left(X\right)$ denote the characteristic polynomial of
the Frobenius acting on $\mathrm{Jac}\left(C\right)\left[l\right]$.
Throughout the paper we will use the notation $\eta\left(\mathbb{F}\right)=\prod_{u=1}^{\infty}\left(1-\left|\mathbb{F}\right|^{-u}\right)$
for any field $\mathbb{F}$.
\begin{thm}
\label{thm: main thm char poly densities}Let $P\left(X\right)=\prod_{i=1}^{s}P_{i}\left(X\right)^{m_{i}}$
where the $P_{i}\left(X\right)\in\mathbb{F}_{l}\left[X\right]$ are
irreducible polynomials which are coprime and such that $l\nmid P\left(q\right)$.
Let $R_{i}=\mathbb{F}_{l}\left[X\right]/\left(P_{i}\left(X\right)^{m_{i}+1}\right)$
and $\mathbb{F}_{i}=R_{i}/\left(P_{i}\left(X\right)\right)$. Assume
$\prod_{i=1}^{s}\eta\left(\mathbb{F}_{i}\right)>1/2$. Let $\epsilon>0$.
Then there exists $N$ such that if $g,q>N$
\[
\left|\mathrm{Prob}\left(P\left(X\right)\mid\mid P_{C}\left(X\right)\right)-\prod_{i=1}^{s}\sum_{\substack{M\in S_{R_{i}}\\
\dim_{\mathbb{F}_{l}}M=m_{i}
}
}\frac{\eta\left(\mathbb{F}_{i}\right)}{\left|\mathrm{Aut}_{R_{i}}\left(M\right)\right|}\right|<\epsilon.
\]
\end{thm}

Note that the sum on the right is finite since ($R_{i}$ is a finite
ring) and can be computed explicitly (see Corollary \ref{cor:explicit formula}).

The condition $\prod_{i=1}^{s}\eta\left(\mathbb{F}_{i}\right)>1/2$
is satisfied when all the $\left|\mathbb{F}_{i}\right|$ are large
enough, which in turn holds if $l$ is large enough, or $\deg P_{i}$
is large enough for all $i$. For example if $s=1$ it holds for all
$l>2$.

Also note that $\mathrm{Prob}\left(P_{C}\left(X\right)=P\left(X\right)\right)=0$
for any fixed $P\left(X\right)\in\mathbb{F}_{l}\left[X\right]$, since
$\deg P_{C}$ goes to infinity, hence reducing to the question of
divisibilty as in the above theorem is necessary to obtain a non-trivial
answer. As an example we have the following Corollary.
\begin{cor}
\label{cor:intro examples}For $l>2$ and any $a\in\mathbb{F}_{l}$
we have
\[
\mathrm{Prob}\left(\left(X-a\right)\mid\mid P_{C}\left(X\right)\right)\longrightarrow\prod_{i=1}^{\infty}\left(1-l^{-i}\right)\left[\frac{1}{\left(l^{2}-1\right)\left(l^{2}-l\right)}+\frac{1}{l^{2}-l}\right]
\]
 and 
\[
\mathrm{Prob}\left(\left(X-a\right)\nmid P_{C}\left(X\right)\right)\longrightarrow\prod_{i=1}^{\infty}\left(1-l^{-i}\right)
\]
 in the $g,q$-limit.
\end{cor}

We also prove the asymptotic independence of the appearance of certain
submodules of $\mathrm{Jac}\left(C\right)\left[l^{\infty}\right]$.
Call $\lambda\in\overline{\mathbb{Q}_{l}}$ a Weil $q$-number if
it satisfies $\left|\phi\left(\lambda\right)\right|=\sqrt{q}$ for
any embedding $\phi:\overline{\mathbb{Q}_{l}}\longrightarrow\mathbb{C}$.
\begin{thm}
\label{thm:main thm indep eigenspaces}Let $\epsilon>0$. Let $P_{i}\left(X\right)\in\mathbb{Z}_{l}\left[X\right]$
be irreducible, pairwise coprime modulo $l$, and $l\nmid P_{i}\left(q\right)$
for $i=1,\ldots,s$. Suppose the roots of $P_{i}\left(X\right)$ are
not Weil $q$-numbers and $\prod_{i=1}^{s}\eta\left(\mathbb{F}_{i}\right)>1/2$.

Let $R=\mathbb{Z}_{l}\left[X\right]/\left(\prod_{i=1}^{s}P_{i}\left(X\right)^{m_{i}}\right)$.
Let $M\in S_{R}$ and let $M_{i}=M/\left(P_{i}\left(X\right)^{m_{i}}\right)M$.
There exists $N$ such that if $g,q>N$ then 
\[
\left|\mathrm{Prob}\left(\mathrm{Jac}\left(C\right)\left[\prod_{i=1}^{s}P_{i}\left(F\right)^{m_{i}}\right]\cong M\right)-\prod_{i=1}^{s}\mathrm{Prob}\left(\mathrm{Jac}\left(C\right)\left[P_{i}\left(F\right)^{m_{i}}\right]\cong M_{i}\right)\right|<\epsilon.
\]
\\
\end{thm}

We now discuss the crucial theorem needed to prove the results stated
above. In Section \ref{sec:The-Cohen-Lenstra-measure} we define a
measure $\mu_{R}$ on the set of $R$-modules $S_{R}$ by taking cokernels
of random matrices over $R$ (see also Theorem \ref{thm:main thm abstract measure}
below). By proving a convergence of moments implies convergence of
measures result (Theorem \ref{thm:conv of moments gives measure})
we can combine it with \cite[Theorem 1.1]{lipnowskitsimerman} to
show that the densities of the cokernel of Frobenius acting on the
Tate module of $C\in\mathcal{M}_{g}\left(\mathbb{F}_{q}\right)$ converge
to $\mu_{R}$.

Let $T_{l}\left(\mathrm{Jac}\left(C\right)\right)$ denot the $l$-adic
Tate module. Let $P\left(X\right)\in\mathbb{Z}_{l}\left[X\right]$
and suppose the roots of $P\left(X\right)$ are not Weil $q$-numbers.
Fix $g,q$ and define the probability measure on $S_{R}$ by 
\[
\mu_{g,q}\left(M\right)=\frac{\left|\left\{ C\in\mathcal{M}_{g}\left(\mathbb{F}_{q}\right)\mid T_{l}\left(\mathrm{Jac}\left(C\right)\right)/\left\langle P\left(F\right)\right\rangle \cong M\right\} \right|}{\left|\mathcal{M}_{g}\left(\mathbb{F}_{q}\right)\right|}.
\]
 Note this measure is defined using the the cokernel of the action
of Frobenius instead of the kernel which we consider above, but by
duality this amounts to the same thing under certain hypothesis (see
Lemma \ref{lem:mu' equals mu}).
\begin{thm}
\label{thm:main thm cokrnel density}Assume $\prod_{i=1}^{s}\eta\left(\mathbb{F}_{i}\right)>1/2$.
The sequence of probability measures $\left\{ \mu_{g,q}\right\} $
converges weakly to $\mu_{R}$ as $g,q\longrightarrow\infty$ along
any sequence with $l\nmid P\left(q\right)$.
\end{thm}

We prove this theorem more generally for any sequence of measures
$\left\{ \nu_{n}\right\} $ on $S_{R}$ whose moments converge to
1, where $R$ is any finite $\mathbb{Z}_{l}$-algebra. This result
is proven by showing that there is no escape of mass along the sequence
$\left\{ \nu_{n}\right\} $ with respect to the integrals $\int_{S_{R}}\mathrm{Surj}_{R}\left(X,A\right)d\nu_{n}\left(X\right)$
for any $A\in S_{R}$. This is done by finding $A'\in S_{R}$ for
which the ratio $\mathrm{Surj}_{R}\left(X,A\right)/\mathrm{Surj}_{R}\left(X,A'\right)$
can be made arbitrarily small and taking advantage of the fact that
the moments are uniformly bounded.

\subsection{A random model}

On the random model side we prove the following theorem on the measure
$\mu_{R}$ which is analogous to the classic Cohen-Lenstra measure,
and generalizes Theorem 1.2 from \cite{lipnowskitsimerman}.

They in fact constructed a more general random model for triples of
the form $\left(G,F,\omega\right)$ where $G$ is a finite $l$-group
on which $F$ acts invertibly, $\omega\in\wedge^{2}G$, and $F\omega=q\omega$
(see Section 1.3.2 of \cite{lipnowskitsimerman}) which is a refinment
of the classical random matrix model of Friedman-Washington \cite{FriedmanWashington}.
They conjectured a refinement of $\mu_{g,q}$, which additionally
accounts for isomorphism of the Weil pairing on $\mathrm{Jac}\left(C\right)$
with $\omega$, converges to the measure given by their random model
in the $g$ limit. However their proofs were restricted to the case
when $\omega=0$ and Weil pairing is trivial, that is the measure
becomes $\mu_{g,q}$. Theorem \ref{thm:main thm cokrnel density}
proves this conjecture in the $g,q$-limit (the condition $l\nmid P\left(q\right)$
implies the Weil pairing is trivial).

There is also the recent work of Cheong-Huang \cite{cheonghuang}
who proved identities related to the measure $\mu_{R}$ in the case
when $R$ is a complete discrete valuation ring.

We let $\left\{ m_{1},\ldots,m_{s}\right\} $ be the maximal ideals
of $R$ and $\mathbb{F}_{i}=R/m_{i}$. For any $M\in S_{R}$ define
$d_{m_{i}}\left(M\right)=\dim_{\mathbb{F}_{i}}M\otimes_{R}\mathbb{F}_{i}-\dim_{\mathbb{F}_{i}}\mathrm{Tor}_{R}^{1}\left(M,\mathbb{F}_{i}\right)$.
Let $T_{R}\subseteq S_{R}$ be the set of $R$-modules which are isomorphic
to the cokernel of some square matrix over $R$. For $\vec{j}\in\mathbb{Z}^{s}$
define the subset $T_{R,\vec{j}}\subset S_{R}$ by 
\begin{equation}
T_{R,\vec{j}}=\left\{ M\in T_{R}\mid d_{m_{i}}\left(M\right)=j_{i}\right\} .\label{eq:partition TR-1}
\end{equation}

\begin{thm}
\label{thm:main thm abstract measure}Assume $\prod_{i=1}^{s}\eta\left(\mathbb{F}_{i}\right)>1/2$.
There exists a unique probability measure $\mu_{R}$ on $S_{R}$ such
that 
\[
\int_{S_{R}}\mathrm{Surj}_{R}\left(X,A\right)d\mu_{R}\left(X\right)=1
\]
 for all $A\in S_{R}$. The measure $\mu_{R}$ is supported on $\bigcup_{\vec{j}\in\mathbb{Z}_{\ge0}^{s}}T_{R,\vec{j}}$
and for every $\vec{j}\in\mathbb{Z}_{\ge0}^{s}$ and $M\in T_{R,\vec{j}}$
we have the formula
\[
\mu_{R}\left(M\right)=\frac{c_{R,\vec{j}}}{\left|\mathrm{Aut}_{R}M\right|}
\]
 where 
\[
c_{R,\vec{j}}=\lim_{n\longrightarrow\infty}\prod_{i=1}^{s}\frac{\mathcal{N}_{i}\left(n,j_{i}\right)|\mathrm{GL}_{n-j_{i}}\left(\mathbb{F}_{i}\right)|}{|M_{n\times n-j_{i}}\left(\mathbb{F}_{i}\right)|}
\]
 and $\mathcal{N}_{i}\left(n,j_{i}\right)$ denotes the number of
subspaces of $\mathbb{F}_{i}^{n}$ of dimension $j_{i}$.

If $\mathbb{Z}_{l}\subset R$ then $\mu_{R}\left(M\right)$ is inversely
proportional to $\left|\mathrm{Aut}_{R}M\right|$ for all $M\in S_{R}$.
\end{thm}

In general $\mu_{R}$ is not inversely proportional to the size of
the automorphism group as the constant $c_{R,\vec{j}}$ depends on
the module $M$. If $\mathbb{Z}_{l}\subset R$ it turns out that $T_{R,\vec{j}}$
is empty unless $\vec{j}=\vec{0}$.

Additionally in Section \ref{subsec:The-distribution-of} we prove
some formulas for the distribution and moments of the function $\mathrm{rk}_{l}$
on $S_{R}$ which are needed to prove our theorems, but which may
be of independent interest.

\section{Measures on $R$-modules}

\subsection{Preliminaries}

Let $R$ be a ring which is a finite $\mathbb{Z}_{l}$-algebra. Let
$J=J\left(R\right)$ be the Jacobson radical of $R$. Let $\left\{ m_{1},\ldots,m_{s}\right\} $
be the set of maximal ideals of $R$. Let $\mathbb{F}_{i}\cong R/m_{i}$
be the residue field of $m_{i}$. By the Chinese remainder theorem
we have $R/J\left(R\right)\cong\prod_{i=1}^{s}\mathbb{F}_{i}$ and
for any finite $R$-module $M$ we have $M/J\left(R\right)M\cong\prod_{i=1}^{s}V_{i}$
where $V_{i}\cong M/m_{i}M$ is a finite-dimensional $\mathbb{F}_{i}$
vector space.

Let $S_{R}$ be the set of finite $R$-modules. For any $e\ge1$ let
$S_{e}\subset S_{R}$ be the set of modules $M$ such that $m_{i}^{e}M=0$
for all $i$. Clearly every finite $R$-module is contained in $S_{e}$
for large enough $e$. We have $S_{e}=\left\{ M\otimes_{R}\prod_{i=1}^{s}R/m_{i}^{e}\mid M\in S_{R}\right\} $.
\begin{lem}
For each $i$ the field $\mathbb{F}_{i}$ is a finite extension of
$\mathbb{F}_{l}$.
\end{lem}

\begin{proof}
Suppose towards a contradiction that $l\notin m_{i}$. Let $\phi:\mathbb{Z}_{l}\longrightarrow R$
defined by $\phi\left(x\right)=x\cdot1_{R}$. Since $\phi^{-1}\left(m_{i}\right)$
is a prime ideal of $\mathbb{Z}_{l}$ it is either 0 or $\left(l\right)$.
Hence in our case $\phi^{-1}\left(m_{i}\right)=0$ so in particular
$\phi$ is an injection and $\phi\left(\mathbb{Z}_{l}\right)\cap m_{i}=0$.
Since $R/m_{i}$ is a field containing $\mathbb{Z}_{l}$ it is an
extension of $\mathbb{Q}_{l}$.

Now $R$ is a Noetherian $\mathbb{Z}_{l}$-module since it is finite
over the Noetherian ring $\mathbb{Z}_{l}$. Hence its quotient $R/m_{i}$
is a Noetherian $\mathbb{Z}_{l}$-module. But this implies the submodule
$\mathbb{Q}_{l}$ of $R/m_{i}$ is also Noetherian, a contradiction.
Thus we conclude $l\in m_{i}$.

Since $R$ is finite over $\mathbb{Z}_{l}$ this implies $\mathbb{F}_{i}=R/m_{i}$
is a finite extension of $\mathbb{F}_{l}=\mathbb{Z}_{l}/\left(l\right)$.
\end{proof}
\begin{lem}
\label{lem:R/m^e finite}For each $i$ the ring $R/m_{i}^{e}$ is
finite.
\end{lem}

\begin{proof}
The field $\mathbb{F}_{i}=R/m_{i}$ is finite since it is a finite
extension of $\mathbb{F}_{l}$. Note $m_{i}^{j}/m_{i}^{j+1}\cong m_{i}^{j}\otimes_{R}R/m_{i}$
and this is a finite dimensional $\mathbb{F}_{i}$-vector space since
$R$ is a Noetherian ring. Then by the exact sequence 
\[
0\longrightarrow m_{i}^{j}/m_{i}^{j+1}\longrightarrow R/m_{i}^{j+1}\longrightarrow R/m_{i}^{j}\longrightarrow0
\]
and induction we get the result.
\end{proof}
For any $X\in S_{R}$ let $r_{i}\left(X\right)=\dim_{\mathbb{F}_{i}}X/m_{i}X$
and let $r\left(X\right)=\sum_{i=1}^{s}r_{i}\left(X\right)$.
\begin{lem}
\label{lem:rank grows with size}Fix $e\in\mathbb{Z}$. For any $r>0$
there exists $n>0$ such if $X\in S_{e}$ and $\left|X\right|>n$
then $r\left(X\right)>r$ .
\end{lem}

\begin{proof}
Since $X\in S_{e}$ we can view it as a $\prod_{i=1}^{s}R/m_{i}^{e}$-module.
Hence by Nakayama's lemma and Lemma \ref{lem:R/m^e finite} for any
$X\in S_{e}$ we have 
\begin{align*}
\left|X\right| & \le\sum_{i=1}^{s}r_{i}\left(X\right)\left|R/m_{i}^{e}\right|\\
 & \le r\left(X\right)\cdot\max_{i}\left|R/m_{i}^{e}\right|.
\end{align*}
.
\end{proof}

\subsection{\label{subsec:Moments-and-limits}Moments and limits of measures
on $R$-modules}

For any $A\in S_{R}$ let $f_{A}\left(X\right)=\left|\mathrm{Surj}_{R}\left(X,A\right)\right|$
for all $X\in S_{R}$.

We first restrict to the set $S_{e}$ and prove the following.
\begin{prop}
\label{prop:surj bound}For any $\epsilon>0$ and $A\in S_{e}$ there
exists $A'\in S_{e}$ and $c>0$ such that 
\[
f_{A}\left(X\right)\le\epsilon f_{A'}\left(X\right)
\]
 for all $X\in S_{e}$ with $\left|X\right|>c$.
\end{prop}

\begin{proof}
Let $A'=A\oplus R/J\left(R\right)$. Recall $R/J\left(R\right)\cong\prod_{i=1}^{s}R/m_{i}$.
For each $f\in\mathrm{Surj}_{R}\left(X,A\right)$ we will construct
a set of surjections from $X$ to $R/J\left(R\right)$ whose size
grows with the rank of $\ker f$.

By Lemma \ref{lem:rank grows with size} for any $N_{0}>0$ we can
take $c$ large enough such that $r\left(\ker f\right)>N_{0}$ for
all $f\in\mathrm{Surj}_{R}\left(X,A\right)$ and all $\left|X\right|>c$.

Fix $X\in S_{e}$ with $\left|X\right|>c$ such that $f_{A}\left(X\right)>0$.
Fix $f\in\mathrm{Surj}_{R}\left(X,A\right)$. The set $\mathrm{Surj}_{R}\left(\ker f,R/J\left(R\right)\right)$
is in bijection with 
\[
\mathrm{Surj}_{R}\left(\ker f\otimes_{R}R/J\left(R\right),R/J\left(R\right)\right).
\]
For any $N_{1}>0$ there exists $N_{0}$ such that if $r\left(\ker f\right)>N_{0}$
then $f_{R/J\left(R\right)}\left(\ker f\otimes_{R}R/J\left(R\right)\right)>N_{1}$
since we are reduced to considering maps of vector spaces. We have
shown that for any $N_{1}>0$ there exists $c>0$ such that for all
$\left|X\right|>c$ and all $f\in\mathrm{Surj}_{R}\left(X,A\right)$
we have $f_{R/J\left(R\right)}\left(\ker f\otimes_{R}R/J\left(R\right)\right)>N_{1}$.

Let $\pi_{X,f}\in\mathrm{Surj}\left(X,\ker f/R\left(J\right)\right)$
be the composition of projections
\[
X\longrightarrow X/R\left(J\right)\cong V\oplus\ker f/R\left(J\right)\longrightarrow\ker f/R\left(J\right).
\]
 The second isomorphism follows since $\ker f/R\left(J\right)\cong\prod_{i=1}^{s}\ker f/m_{i}$
and $\ker f/m_{i}$ is a subspace of $X/m_{i}X$ hence a direct summand.
Given any $g'\in\mathrm{Surj}_{R}\left(\ker f/R\left(J\right),R/J\left(R\right)\right)$
we get $g\in\mathrm{Surj}_{R}\left(X,R/m\right)$ by letting $g=g'\circ\pi_{X,f}$.
Letting 
\[
G\left(X,f\right)=\left\{ g=g'\circ\pi_{X,f}\mid g'\in\mathrm{Surj}_{R}\left(\ker f/R\left(J\right),R/J\left(R\right)\right)\right\} \subset\mathrm{Surj}_{R}\left(X,R/J\left(R\right)\right)
\]
 we have $\left|G\left(X,f\right)\right|>N_{1}$. Thus for any $N_{1}>0$
there exists $c$ such that if $\left|X\right|>c$ then $\left|G\left(X,f\right)\right|>N_{1}$
for all $f\in\mathrm{Surj}_{R}\left(X,A\right)$.

We claim that 
\[
T\left(X\right)=\left\{ \left(f,g\right)\mid f\in\mathrm{Surj}_{R}\left(X,A\right),g\in G\left(X,f\right)\right\} \subseteq\mathrm{Surj}_{R}\left(X,A'\right)
\]
 for all $X\in S_{e}$. Note for any $\left(f,g\right)\in T\left(X\right)$
we have that $g\mid_{\ker f}$ is already a surjection to $R/J\left(R\right)$.
If $\left(x,a\right)\in A'=A\oplus R/J\left(R\right)$ then pick $y_{1}\in X$
such that $f\left(y_{1}\right)=x$ and $y_{2}\in\ker f$ such that
$g\left(y_{2}\right)=a-g\left(y_{1}\right)$. Then $\left(f,g\right)\left(y_{1}+y_{2}\right)=\left(x,a\right)$.

We conclude that for any $N_{1}>0$ we can take $c$ large enough
such that $\left|T\left(X\right)\right|>N_{1}\cdot f_{A}\left(X\right)$
for all $\left|X\right|>c$. This completes the proof.
\end{proof}
For the next lemma make the notation $Y_{c}=\left\{ X\in S_{R}\mid\left|X\right|>c\right\} $.
\begin{lem}
\label{lem:bounded integrals}Let $e\in\mathbb{Z}$. For any $\epsilon>0$
and $A\in S_{e}$ there exists $c>0$ and $A'\in S_{e}$ such that
\[
\int_{Y_{c}}f_{A}\left(X\right)d\nu<\epsilon\int_{Y_{c}}f_{A'}\left(X\right)d\nu
\]
 for all measures $\nu$ on $S_{R}$.
\end{lem}

\begin{proof}
For any $A\in S_{e}$ and $X\in S_{R}$ we have $\mathrm{Surj}_{R}\left(X,A\right)=\mathrm{Surj}_{R}\left(X\otimes\prod_{i=1}^{s}R/m_{i}^{e},A\right)$.
Hence by Proposition \ref{prop:surj bound} there exists $c>0$ and
$A'\in S_{e}$ such that

\begin{align*}
\int_{Y_{c}}f_{A}\left(X\right)d\nu & =\int_{Y_{c}}f_{A}\left(X\otimes\prod_{i=1}^{s}R/m_{i}^{e}\right)d\nu\\
 & <\epsilon\int_{Y_{c}}f_{A'}\left(X\otimes\prod_{i=1}^{s}R/m_{i}^{e}\right)d\nu\\
 & =\epsilon\int_{Y_{c}}f_{A'}\left(X\right)d\nu
\end{align*}
for any measure $\nu$.
\end{proof}
We now prove a ``convergence of moments'' implies ``convergence
of measures'' result. We make use of the above results to ensure
no escape of mass for the sequence of measures in question.
\begin{thm}
\label{thm:conv of moments gives measure}Let $\left\{ \nu_{n}\right\} $
be a sequence of probability measures on $S_{R}$ such that 
\[
\lim_{n\longrightarrow\infty}\int_{S_{R}}f_{A}\left(X\right)d\nu_{n}=1
\]
 for all $A\in S_{R}$. Then the sequence $\left\{ \nu_{n}\right\} $
converges weakly to a probability measure $\mu$ satisfying $\int_{S_{R}}f_{A}\left(X\right)d\mu=1$
for all $A\in S_{R}$.
\end{thm}

\begin{proof}
Note that by restricting to a subsequence we may assume the $\nu_{n}$
weakly converge to some measure $\nu_{\infty}$ (which may not be
a probability measure).

For any $c>0$
\begin{align*}
\int_{S_{R}}f_{A}\left(X\right)d\nu_{\infty} & =\int_{\left|X\right|<c}f_{A}\left(X\right)d\nu_{\infty}+\int_{\left|X\right|>c}f_{A}\left(X\right)d\nu_{\infty}\\
 & \ge\int_{\left|X\right|<c}f_{A}\left(X\right)d\nu_{\infty}.
\end{align*}
 This last integral is a finite sum, hence by weak convergence of
the $\nu_{n}$ we have 
\begin{align*}
\int_{\left|X\right|<c}f_{A}\left(X\right)d\nu_{\infty} & =\lim_{n\longrightarrow\infty}\int_{\left|X\right|<c}f_{A}\left(X\right)d\nu_{n}\\
 & =1-\lim_{n\longrightarrow\infty}\int_{\left|X\right|>c}f_{A}\left(X\right)d\nu_{n}.
\end{align*}
Thus for any $c>0$
\[
\int_{S_{R}}f_{A}\left(X\right)d\nu_{\infty}\ge1-\lim_{n\longrightarrow\infty}\int_{\left|X\right|>c}f_{A}\left(X\right)d\nu_{n}.
\]
Then by Lemma \ref{lem:bounded integrals} for any $\epsilon>0$ there
exist $c>0$ and $A'$ such that 
\begin{align*}
\lim_{n\longrightarrow\infty}\int_{\left|X\right|>c}f_{A}\left(X\right)d\nu_{n} & \le\epsilon\lim_{n\longrightarrow\infty}\int_{\left|X\right|>c}f_{A'}\left(X\right)d\nu_{n}\\
 & \le\epsilon.
\end{align*}
 Thus  $\int_{S_{R}}f_{A}\left(X\right)d\nu_{\infty}\ge1-\epsilon$
for any $\epsilon>0$. By Fatou's Lemma for measures 
\begin{align*}
\int_{S_{R}}f_{A}\left(X\right)d\nu_{\infty} & \le\lim_{n\longrightarrow\infty}\int_{S_{R}}f_{A}\left(X\right)d\nu_{n}\\
 & =1.
\end{align*}
 Thus $\int_{S_{R}}f_{A}\left(X\right)d\nu_{\infty}=1$ for all $A\in S_{R}$.
Setting $A=1$ shows $\nu_{\infty}$ is in fact a probability measure.
\end{proof}

\section{\label{sec:The-Cohen-Lenstra-measure}The Cohen-Lenstra measure on
$R$-modules}

We now define a particular measure on $R$-modules using the standard
method of taking cokernels of large random matrices. We prove various
properties of this measure, in particular that it is determined by
its moments. This generalizes \cite[Section 2]{lipnowskitsimerman}
to rings $R$ which are finite $\mathbb{Z}_{l}$-algebras. Some of
the arguments are similar but we reproduce them here for completeness
and since additional work is required to deal with the more general
rings in our case.

Let $R$ and $S_{R}$ be as defined above. Let $\mu_{\mathrm{haar},n}$
be the Haar measure induced on $\mathrm{End}_{R}\left(R^{n}\right)$
(since $R$ is finite over $\mathbb{Z}_{l}$) and let $\phi_{n}:\mathrm{End}_{R}\left(R^{n}\right)\longrightarrow S_{R}$
be defined by $\phi_{n}\left(f\right)=\mathrm{coker}f$. For each
$n$ we have a probability measure on $S_{R}$ given by $\mu_{n}=\left(\phi_{n}\right)_{*}\mu_{\mathrm{haar},n}$.
Finally let $\mu_{R}=\lim_{n\longrightarrow\infty}\mu_{n}$ be the
weak limit of measures.

We start by showing that $\mu_{R}$ is indeed a probability measure.
\begin{lem}
\label{lem:measure moments are 1}The measure $\mu_{R}$ is a probability
measure on $S_{R}$ and $\int_{S_{R}}f_{A}\left(X\right)d\mu_{R}=1$
for all $A\in S_{R}$.
\end{lem}

\begin{proof}
We claim $\lim_{n\longrightarrow\infty}\int_{S_{R}}f_{A}\left(X\right)d\mu_{n}=1$
for all $A\in S_{R}$. For any $n$ by definition of $\mu_{n}$ as
the push-forward we have 
\begin{align*}
\int_{S_{R}}f_{A}\left(X\right)d\mu_{n} & =\int_{\varphi\in\mathrm{End}_{R}\left(R^{n}\right)}f_{A}\left(\mathrm{coker}\varphi\right)d\mu_{\mathrm{haar},n}\\
 & =\sum_{\psi\in\mathrm{Surj}\left(R^{n},A\right)}\mu_{\mathrm{haar},n}\left(\left\{ \varphi\in\mathrm{End}_{R}\left(R^{n}\right)\mid\mathrm{im}\varphi\subseteq\ker\psi\right\} \right)\\
 & =\left|\mathrm{Surj}\left(R^{n},A\right)\right|\cdot\left|A\right|^{-n}
\end{align*}
 where we are counting surjections to $A$ in two different ways.
The last equality follows since $\ker\psi$ has index $\left|A\right|$
in $R^{n}$. Then $\lim_{n\longrightarrow\infty}\left|\mathrm{Surj}\left(R^{n},A\right)\right|=\left|\mathrm{Hom}\left(R^{n},A\right)\right|=\left|A\right|^{n}$.
This proves the claim. Hence the result follows by Theorem \ref{thm:conv of moments gives measure}.
\end{proof}
Now we characterize the support of $\mu_{R}$ under some additional
assumptions on $R$. We will see the support of $\mu_{R}$ is a strict
subset of $S_{R}$.

Let $T_{R}$ be the union of $\mathrm{im}\phi_{n}$ for all $n$.

Recall we let $\left\{ m_{1},\ldots,m_{s}\right\} $ be the maximal
ideals of $R$ and $\mathbb{F}_{i}=R/m_{i}$. For any $M\in S_{R}$
define $d_{m_{i}}\left(M\right)=\dim_{\mathbb{F}_{i}}M\otimes_{R}\mathbb{F}_{i}-\dim_{\mathbb{F}_{i}}\mathrm{Tor}_{R}^{1}\left(M,\mathbb{F}_{i}\right)$.
Let $d\left(M\right)=\sum_{i=1}^{s}d_{m_{i}}\left(M\right)$.
\begin{lem}
\label{lem:characterize T_R}Let $M\in S_{R}$. Then $M\in T_{R}$
if and only if $d_{i}\left(M\right)\ge0$ for all $i$.
\end{lem}

\begin{proof}
Since $M$ is finite there is an exact sequence
\begin{equation}
0\longrightarrow K\longrightarrow R^{n}\longrightarrow M\longrightarrow0\label{eq:T_R exact}
\end{equation}
 which we can localize at any maximal ideal $m$ of $R$. Let $\mathbb{F}=R/m$.
Note $\dim_{\mathbb{F}}\mathrm{Tor}_{R_{m}}^{1}\left(R_{m}^{n},\mathbb{F}\right)$
$=0$. Hence tensoring with $\mathbb{\mathbb{F}}\cong R_{m}/mR_{m}$
we see that 
\begin{equation}
\dim_{\mathbb{F}}M_{m}\otimes_{R_{m}}\mathbb{F}-\dim_{\mathbb{F}}\mathrm{Tor}_{R_{m}}^{1}\left(M_{m},\mathbb{F}\right)=n-\dim_{\mathbb{F}}K_{m}\otimes_{R_{m}}\mathbb{F}.\label{eq:T_R equation}
\end{equation}

Suppose $M\in T_{R}$. Then there is an exact sequence 
\[
0\longrightarrow N\longrightarrow R^{n}\longrightarrow R^{n}\longrightarrow M\longrightarrow0.
\]
 Let $K=R^{n}/N$. Then $\dim_{\mathbb{F}}R^{n}/N\otimes_{R}\mathbb{F}\le n$
so it follows from (\ref{eq:T_R equation}) that $d_{m_{i}}\left(M\right)\ge0$
for all $i$.

Conversely suppose $d_{i}\left(M\right)\ge0$. Let $n$ and $K$ be
as in (\ref{eq:T_R exact}). Then $\dim_{\mathbb{F}_{i}}K\otimes_{R}\mathbb{F}_{i}\le n$
for all $i$. Thus $K/J\left(K\right)\cong\prod_{i=1}^{s}\mathbb{F}_{i}^{n_{i}}$
for $n_{i}\le n$ and consequently by Nakayama's lemma $K$ is generated
by at most $n$ elements over $R$. Hence there is a surjective map
$R^{n}\longrightarrow K$ so $K\cong R^{n}/N$ for some submodule
$N$. Thus $M$ is the cokernel of a map $R^{n}\longrightarrow R^{n}$
so $M\in T_{R}$. This completes the proof.

\end{proof}
We now give a specific formula for $\mu_{R}$ analogous to the classic
Cohen-Lenstra measure. Let $T_{R}\subset S_{R}$ be the set of modules
in the image of $\phi_{n}$ for at least one $n$. Recall $R/J\left(R\right)\cong\prod_{i=1}^{s}\mathbb{F}_{i}$.

For $\vec{j}\in\mathbb{Z}_{\ge0}^{s}$ define the subset $T_{R,\vec{j}}\subset T_{R}$
by 
\begin{equation}
T_{R,\vec{j}}=\left\{ M\in T_{R}\mid d_{m_{i}}\left(M\right)=j_{i}\right\} .\label{eq:partition TR}
\end{equation}
 Clearly we can partition $T_{R}$ into the disjoint union $T_{R}=\bigcup_{\vec{j}\in\mathbb{Z}^{s}}T_{R,\vec{j}}$
(though some of the $T_{R,\vec{j}}$ may be empty).
\begin{thm}
\label{thm:The-measure-}The measure $\mu_{R}$ is supported on $T_{R}$
and for every $\vec{j}\in\mathbb{Z}_{\ge0}^{s}$ and $M\in T_{R,\vec{j}}$
we have the formula
\[
\mu_{R}\left(M\right)=\frac{c_{R,\vec{j}}}{\left|\mathrm{Aut}_{R}M\right|}
\]
 where 
\[
c_{R,\vec{j}}=\lim_{n\longrightarrow\infty}\prod_{i=1}^{s}\frac{\mathcal{N}_{i}\left(n,j_{i}\right)|\mathrm{GL}_{n-j_{i}}\left(\mathbb{F}_{i}\right)|}{|M_{n\times n-j_{i}}\left(\mathbb{F}_{i}\right)|}
\]
 and $\mathcal{N}_{i}\left(n,j_{i}\right)$ denotes the number of
subspaces of $\mathbb{F}_{i}^{n}$ of dimension $j_{i}$.
\end{thm}

\begin{rem}
\label{rem:C_r}Note that Theorem \ref{thm:The-measure-} shows that
the formula for $\mu_{R}$ is not always inversely proportional to
the size of the automorphism group as the constant $c_{R,\vec{j}}$
depends on the module $M$. See Lemma \ref{lem:cR formula} for an
explicit formula for $c_{R,\vec{j}}$.
\end{rem}

\begin{proof}
Clearly $\mu_{R}$ is only supported on $T_{R}$ since the $\mu_{n}$
are. Let $M\in T_{R}$.

Let 
\[
X=\left\{ \left(\varphi,\psi\right)\in\mathrm{End}_{R}\left(R^{n}\right)\times\mathrm{Surj}_{R}\left(R^{n},M\right)\mid\mathrm{im}\varphi=\ker\psi\right\} 
\]
and let $\mu$ denote the product measure on $\mathrm{End}_{R}\left(R^{n}\right)\times\mathrm{Surj}_{R}\left(R^{n},M\right)$
given by the product of $\mu_{\mathrm{haar},n}$ with the counting
measure. We will compute the measure of $X$ in two different ways.
Firstly by definition
\[
\mu_{\mathrm{haar},n}\left(\left\{ \varphi\in\mathrm{End}_{R}\left(R^{n}\right)\mid\mathrm{coker}\varphi\cong M\right\} \right)=\mu_{R,n}\left(M\right).
\]
For each $\varphi\in\mathrm{End}_{R}\left(R^{n}\right)$ the group
$\mathrm{Aut}\left(M\right)$ acts freely faithfully on the set 
\[
X_{\varphi}=\left\{ \psi\in\mathrm{Surj}_{R}\left(R^{n},M\right)\mid\left(\varphi,\psi\right)\in X\right\} .
\]
Thus $\mu\left(X\right)=\mu_{R,n}\left(M\right)\left|\mathrm{Aut}\left(M\right)\right|$.

We can also express $\mu\left(X\right)$ as

\begin{align*}
\mu\left(X\right) & =\sum_{\psi\in\mathrm{Surj}_{R}\left(R^{n},M\right)}\mu_{\mathrm{haar},n}\left(\left\{ \varphi\in\mathrm{End}_{R}\left(R^{n}\right)\mid\mathrm{im}\varphi=\ker\psi\right\} \right)\\
 & =\sum_{\psi\in\mathrm{Surj}_{R}\left(R^{n},M\right)}\mu_{\mathrm{haar},n}\left(\mathrm{Surj}_{R}\left(R^{n},\ker\psi\right)\right).
\end{align*}
Now fix any $\psi\in\mathrm{Surj}_{R}\left(R^{n},M\right)$.

For some large enough fixed $m$ there exists $f\in\mathrm{End}_{R}\left(R^{m}\right)$
such that $M\cong\mathrm{coker}f$. Let $g:R^{m}\longrightarrow R^{m}/\mathrm{im}f$
be the projection. Since $M\in T_{R,\vec{j}}$ we have $m-\dim_{\mathbb{F}_{i}}\ker g\otimes_{R}\mathbb{F}_{i}=j_{i}$
for all $i$.

Let $\left\{ e_{i}\right\} $ be the standard basis for $R^{n}$.
First suppose there is a subset $V\subset\left\{ e_{i}\right\} $
of size $m$ such that $\psi\mid_{\left\langle V\right\rangle }=g$.
Then we can find an invertible transformation of $R^{n}$ under which
$R^{n}\cong\left\langle V\right\rangle \oplus W$ with $W\cong R^{n-m}$
and $\psi\left(W\right)=0$. Thus in this case we have $\ker\psi\cong\ker g\oplus R^{n-m}$.
We can find $n$ large enough such that the proportion of $\psi\in\mathrm{Surj}_{R}\left(R^{n},M\right)$
for which this holds is arbitrarily close to 1.

Note 
\begin{align*}
\mu_{\mathrm{haar},n}\left(\mathrm{Hom}_{R}\left(R^{n},\ker\psi\right)\right) & =\mu_{\mathrm{haar},n}\left(\left\{ \varphi\in\mathrm{End}_{R}\left(R^{n}\right)\mid\mathrm{im}\varphi\subseteq\ker\psi\right\} \right)\\
 & =\left|M\right|^{-n}
\end{align*}
and $\left|\mathrm{Hom}_{R}\left(R^{n},M\right)\right|=\left|M\right|^{n}$.
Also $\lim_{n\longrightarrow\infty}\left|\mathrm{Surj}_{R}\left(R^{n},M\right)\right|/\left|\mathrm{Hom}_{R}\left(R^{n},M\right)\right|=1$.

Let $K=\ker g\oplus R^{n-m}$. Combining the above facts we see that
\begin{align*}
\lim_{n\longrightarrow\infty}\mu\left(X\right) & =\lim_{n\longrightarrow\infty}\left|M\right|^{n}\cdot\mu_{\mathrm{haar},n}\left(\mathrm{Surj}_{R}\left(R^{n},\ker g\oplus R^{n-m}\right)\right)\\
 & =\lim_{n\longrightarrow\infty}\mu_{\mathrm{haar},n}\left(\mathrm{Surj}_{R}\left(R^{n},K\right)\right)/\mu_{\mathrm{haar},n}\left(\mathrm{Hom}_{R}\left(R^{n},K\right)\right).
\end{align*}

Tensoring with $R/J\left(R\right)$ induces a map 
\[
\mathrm{Hom}_{R}\left(R^{n},K\right)\longrightarrow\mathrm{Hom}_{R/J\left(R\right)}\left(\left(R/J\left(R\right)\right)^{n},K\otimes R/J\left(R\right)\right)
\]
 which by Nakayama's lemma restricts to a map between the corresponding
sets of surjections. Since $\mu_{\mathrm{haar},n}$ pushes forward
to the uniform measure on $\mathrm{End}_{R/J\left(R\right)}\left(\left(R/J\left(R\right)\right)^{n}\right)$
we have 
\[
\mu_{\mathrm{haar},n}\left(\mathrm{Hom}_{R}\left(R^{n},K\right)\right)=\left|\mathrm{Hom}_{R/J\left(R\right)}\left(\left(R/J\left(R\right)\right)^{n},K\otimes R/J\left(R\right)\right)\right|
\]
 and similarly for the corresponding sets of surjections. 

Recall $K=\ker g\oplus R^{n-m}$ and $M\in T_{R,\vec{j}}$. Hence
$K\otimes R/J\left(R\right)\cong\prod_{i=1}^{s}\mathbb{F}_{i}^{m-j_{i}}\times\prod_{i=1}^{s}\mathbb{F}_{i}^{n-m}$.
We have shown that 
\[
K\otimes R/J\left(R\right)\cong\prod_{i=1}^{s}\mathbb{F}_{i}^{n-j_{i}}.
\]
 Hence we have 
\begin{align*}
\mathrm{Hom}_{R/J\left(R\right)}\left(\left(R/J\left(R\right)\right)^{n},K\otimes R/J\left(R\right)\right) & =\prod_{i=1}^{s}M_{n\times n-j_{i}}\left(\mathbb{F}_{i}\right)
\end{align*}
and 
\begin{align*}
\mathrm{Surj}_{R/J\left(R\right)}\left(\left(R/J\left(R\right)\right)^{n},K\otimes R/J\left(R\right)\right) & =\prod_{i=1}^{s}\mathrm{Surj}\left(\mathbb{F}_{i}^{n},\mathbb{F}_{i}^{n-j_{i}}\right).
\end{align*}

Note $\left|\mathrm{Surj}\left(\mathbb{F}_{i}^{n},\mathbb{F}_{i}^{n-j_{i}}\right)\right|=\mathcal{N}\left(n,j_{i}\right)\left|\mathrm{GL}_{n-j_{i}}\left(\mathbb{F}_{i}\right)\right|$
where $\mathcal{N}\left(n,j_{i}\right)$ is the number of subspaces
of $\mathbb{F}_{i}^{n}$ of dimension $j_{i}$. We conclude 
\[
\lim_{n\longrightarrow\infty}\mu\left(X\right)=\lim_{n\longrightarrow\infty}\prod_{i=1}^{s}\frac{\mathcal{N}_{i}\left(n,j_{i}\right)|\mathrm{GL}_{n-j_{i}}\left(\mathbb{F}_{i}\right)|}{|M_{n\times n-j_{i}}\left(\mathbb{F}_{i}\right)|}
\]
which completes the proof.
\end{proof}
The constant $c_{R,\vec{j}}$ defined in Theorem \ref{thm:The-measure-}
can be computed explicitly.
\begin{lem}
\label{lem:cR formula}Suppose $R/J\left(R\right)\cong\prod_{i=1}^{s}\mathbb{F}_{i}.$
For any $\vec{j}\in\mathbb{Z}_{\ge0}^{s}$ we have
\[
c_{R,\vec{j}}=\prod_{i=1}^{s}\eta\left(\mathbb{F}_{i}\right)\frac{\left|\mathbb{F}_{i}\right|^{j_{i}\left(j_{i}+1\right)/2}}{\prod_{k=1}^{j_{i}}\left(\left|\mathbb{F}_{i}\right|^{k}-1\right)}
\]
\end{lem}

\begin{proof}
Recall we let 
\begin{align*}
c_{R,\vec{j}} & =\lim_{n\longrightarrow\infty}\prod_{i=1}^{s}\frac{\mathcal{N}_{i}\left(n,j_{i}\right)|\mathrm{GL}_{n-j_{i}}\left(\mathbb{F}_{i}\right)|}{|M_{n\times n-j_{i}}\left(\mathbb{F}_{i}\right)|}\\
 & =\lim_{n\longrightarrow\infty}\prod_{i=1}^{s}\frac{\mathcal{N}_{i}\left(n,j_{i}\right)}{\left|\mathbb{F}_{i}\right|^{\left(n-j_{i}\right)j_{i}}}\frac{|\mathrm{GL}_{n-j_{i}}\left(\mathbb{F}_{i}\right)|}{|M_{n-j_{i}}\left(\mathbb{F}_{i}\right)|}.
\end{align*}
There is the well known formula 
\[
\mathcal{N}_{i}\left(n,j_{i}\right)=\frac{\prod_{k=n-j_{i}+1}^{n}\left(\left|\mathbb{F}_{i}\right|^{k}-1\right)}{\prod_{k=1}^{j_{i}}\left(\left|\mathbb{F}_{i}\right|^{k}-1\right)}.
\]
 Note the denominator does not depend on $n$. We have 
\begin{align*}
\lim_{n\longrightarrow\infty}\frac{\prod_{k=n-j_{i}+1}^{n}\left(\left|\mathbb{F}_{i}\right|^{k}-1\right)}{\left|\mathbb{F}_{i}\right|^{\left(n-j_{i}\right)j_{i}}} & =\lim_{n\longrightarrow\infty}\frac{\prod_{k=n-j_{i}+1}^{n}\left|\mathbb{F}_{i}\right|^{k}}{\left|\mathbb{F}_{i}\right|^{\left(n-j_{i}\right)j_{i}}}\\
 & =\lim_{n\longrightarrow\infty}\frac{\left|\mathbb{F}_{i}\right|^{nj_{i}-j_{i}^{2}/2+j_{i}/2}}{\left|\mathbb{F}_{i}\right|^{\left(n-j_{i}\right)j_{i}}}\\
 & =\left|\mathbb{F}_{i}\right|^{j_{i}^{2}/2+j_{i}/2}.
\end{align*}

It can be shown that 
\[
\lim_{n\longrightarrow\infty}\left|\mathrm{GL}_{n}\left(\mathbb{F}_{i}\right)\right|/\left|M_{n}\left(\mathbb{F}_{i}\right)\right|=\eta\left(\mathbb{F}_{i}\right).
\]
 This completes the proof.
\end{proof}
As mentioned in Remark \ref{rem:C_r} $\mu_{R}$ is not inversely
proportional to $\left|\mathrm{Aut}_{R}\left(M\right)\right|$ in
general. We give a condition for when this does hold.
\begin{lem}
\label{lem:ring condition}If $R$ contains $\mathbb{Z}_{l}$ then
$d_{i}\left(S_{R}\right)\le0$ for all $i$.
\end{lem}

\begin{proof}
Let $M\in S_{R}$. As in the proof of Lemma \ref{lem:characterize T_R}
we have the sequence
\begin{equation}
0\longrightarrow K\longrightarrow R^{n}\longrightarrow M\longrightarrow0\label{eq:T_R exact-1}
\end{equation}
 which we can localize at any maximal ideal $m$ of $R$ and obtain
the equation $d_{i}\left(M\right)=n-\dim_{\mathbb{F}}K_{m}\otimes_{R_{m}}\mathbb{F}.$

Since $R$ contains $\mathbb{Z}_{l}$ we have $R_{m}\otimes_{\mathbb{Z}_{l}}\mathbb{Q}_{l}\neq0$
and hence $\left(R_{m}\otimes_{\mathbb{Z}_{l}}\mathbb{Q}_{l}\right)^{n}$
requires at least $n$ elements to generate over $R_{m}$.

Now $M_{m}\otimes_{\mathbb{Z}_{l}}\mathbb{Q}_{l}=0$. Hence tensoring
with $\mathbb{Q}_{l}$ we see that a minimal generating set for $K_{m}\otimes_{\mathbb{Z}_{l}}\mathbb{Q}_{l}\cong\left(R_{m}\otimes_{\mathbb{Z}_{l}}\mathbb{Q}_{l}\right)^{n}$
over $R_{m}$ has at least $n$ elements. Hence by Nakayama's lemma
$n\le\dim_{\mathbb{F}}K_{m}\otimes_{R_{m}}\mathbb{F}=\dim_{\mathbb{F}}K\otimes_{R}\mathbb{F}$.
This holds for all maximal ideals $m$ of $R$. Thus $d_{i}\left(M\right)\le0$.
\end{proof}
\begin{cor}
\label{cor:muR formula special case}If $R$ contains $\mathbb{Z}_{l}$
then $\mu_{R}$ is supported on $T_{R}$ and
\[
\mu_{R}\left(M\right)=\frac{c_{R}}{\left|\mathrm{Aut}_{R}\left(M\right)\right|}
\]
 where $c_{R}=\prod_{i=1}^{s}\eta\left(\mathbb{F}_{i}\right)$.
\end{cor}

\begin{proof}
Combining Lemma \ref{lem:characterize T_R} and \ref{lem:ring condition}
we see that $d_{i}\left(S_{R}\right)=0$. Thus $T_{R}=T_{R,\vec{0}}$
and the result follows from Theorem \ref{thm:The-measure-}.
\end{proof}
The following is Lemma 2.4 from \cite{lipnowskitsimerman} the proof
of which works identically in our case. It shows that $\mu_{R}$ is
determined by its moments under an assumption on the normalizing constant
$c_{R}$.
\begin{lem}
\label{lem:mu_R unique}Suppose $c_{R}>1/2$. Let $\nu:S_{R}\longrightarrow\mathbb{R}$
be a positive function satisfying 
\[
\int_{S_{R}}\left|\mathrm{Surj}_{R}\left(B,A\right)\right|d\nu\left(B\right)=1
\]
 for all $A\in S_{R}$. Then $\nu=\mu_{R}$.
\end{lem}

Thus we have shown that $\mu_{R}$ satisfies all of the same familiar
properties as the classical Cohen-Lenstra measure on finite abelian
$l$-groups.

\subsection{\label{subsec:The-distribution-of}The distribution of the function
$\mathrm{rk}_{l}$ on $S_{R}$}

Let $R_{0}=R/lR$. Since $R$ is finite over $\mathbb{Z}_{l}$ this
implies $R_{0}$ is finite.
\begin{lem}
\label{lem:prob reduced mod p}For any $m'$ we have 
\[
\mu_{R}\left(\left\{ M\in S_{R}\mid\mathrm{rk}_{l}M/lM=m'\right\} \right)=\mu_{R_{0}}\left(\left\{ M\in S_{R_{0}}\mid\mathrm{rk}_{l}M=m'\right\} \right).
\]
\end{lem}

\begin{proof}
By definition for any $M\in S_{R}$ we have $\mu_{R}\left(M\right)=\lim_{n\longrightarrow\infty}\mu_{R,n}\left(M\right)$
and 
\[
\mu_{R,n}\left(M\right)=\mu_{\mathrm{haar}}\left(\left\{ N\in\mathrm{Mat}_{n}\left(R\right)\mid\mathrm{coker}N=M\right\} \right).
\]
Since the pushforward of $\mu_{\mathrm{haar}}$ by the reduction $\mathbb{Z}_{l}\longrightarrow\mathbb{F}_{l}$
is the uniform measure (and consequently the same is true for $\mathrm{End}\left(R^{n}\right)\longrightarrow\mathrm{End}\left(R_{0}^{n}\right)$)
we have for any $M_{0}\in S_{R_{0}}$
\begin{align*}
\mu_{R_{0},n}\left(M_{0}\right) & =\left|R_{0}\right|^{-n^{2}}\cdot\left|\left\{ \phi\in\mathrm{End}\left(R_{0}^{n}\right)\mid\mathrm{coker}\phi=M_{0}\right\} \right|\\
 & =\mu_{\mathrm{haar}}\left(\bigcup_{\substack{M\in S_{R}\\
M/lM\cong M_{0}
}
}\left\{ \phi\in\mathrm{End}\left(R^{n}\right)\mid\mathrm{coker}\phi=M\right\} \right)\\
 & =\sum_{\substack{M\in S_{R}\\
M/lM\cong M_{0}
}
}\mu_{R,n}\left(M\right).
\end{align*}
The result follows by taking the limit in $n$ and partitioning the
sum over $M\in S_{R}$ into their reductions modulo $l$.
\end{proof}
Note the set on the right-hand side in Lemma \ref{lem:prob reduced mod p}
is finite. We can give an even more explicit formula for this sum
when $R_{0}$ is a quotient of a Euclidean domain, using the structure
theorem and the formula for the size of the automorphism group $\mathrm{Aut}_{R}\left(M\right)=\mathrm{Aut}_{R_{0}}\left(M\right)$
given in \cite{sizeofautomorphismgroup} (stated there for finite
abelian groups, but the proof works for any Euclidean domain).

For simplicity we will assume for the remainder of this section that
$R$ is local which will suffice for our applications. Since by assumption
$R_{0}$ is a finite discrete valuation ring there exists $m$ such
that $\pi^{m}R_{0}=0$ where $\left(\pi\right)$ is the maximal ideal
of $R_{0}$. Let $\mathcal{P}\left(m'\right)$ be the set of partitions
of the integer $m'$, that is tuples of positive integers for which
the sum of the co-ordinates is $m'$. For any positive integer $j$
we define the subset $\mathcal{P}_{j}\left(m'\right)\subset\mathcal{P}\left(m'\right)$
\[
\mathcal{P}_{j}\left(m'\right)=\left\{ \lambda\in\mathcal{P}\left(m'\right)\mid\text{exactly \ensuremath{j} co-ordinates of \ensuremath{\lambda} are equal to \ensuremath{m}}\right\} .
\]

\begin{cor}
\label{cor:explicit formula}Suppose $R$ is local and $R_{0}$ is
a quotient of a Euclidean domain. Let $\mathbb{F}$ be the residue
field of $R_{0}$ and let $\left|\mathbb{F}\right|=l^{k}$. Then 
\begin{align*}
\mu_{R}\left(\left\{ M\in S_{R}\mid\mathrm{rk}_{l}M/lM=m'\right\} \right) & =\sum_{j=0}^{m'/k}c_{R,j}\sum_{\lambda\in\mathcal{P}_{j}\left(m'/k\right)}\prod_{k=1}^{n_{\lambda}}\left(\left|\mathbb{F}\right|^{d_{k}}-\left|\mathbb{F}\right|^{k-1}\right)^{-1}\prod_{j=1}^{n_{\lambda}}\left(\left|\mathbb{F}\right|^{\lambda_{j}}\right)^{-n_{\lambda}+d_{j}}\\
 & \qquad\times\prod_{i=1}^{n_{\lambda}}\left(\left|\mathbb{F}\right|^{\lambda_{i}-1}\right)^{-n_{\lambda}+c_{i}-1}
\end{align*}
 where for each $\lambda\in\mathcal{P}\left(m'\right)$ of the form
$\left(\lambda_{1},\ldots,\lambda_{n_{\lambda}}\right)$ we define
$d_{k}=\max\left\{ r\mid\lambda_{r}=\lambda_{k}\right\} $ and $c_{k}=\min\left\{ r\mid\lambda_{r}=\lambda_{k}\right\} $.
\end{cor}

\begin{proof}
We reduce to a finite sum over $R_{0}$-modules by Lemma \ref{lem:prob reduced mod p}.
Note since $R_{0}$ is a quotient of a PID every $R_{0}$-module $M$
is of the form 
\[
M\cong\bigoplus_{i=1}^{n}R_{0}/\left(\pi\right)^{n_{i}}
\]
 and if $\mathrm{rk}_{l}M=m'$ this implies $n\le m'/k$.

Recall the support of $\mu_{R_{0}}$ is a disjoint union $T_{R_{0}}=\bigcup_{j=0}^{\infty}T_{R_{0},j}$
(see (\ref{eq:partition TR})). In this case, since $R_{0}/\left(\pi\right)^{m}\cong R_{0}$,
it is easy to see that $M\in T_{R_{0},j}$ is equivalent to exactly
$j$ of the $n_{i}$ being equal to $m$.
\end{proof}
We can also compute the moments of this distribution.

We start by recalling some definitions. Henceforth by a partition
$\mu=\left(\mu_{1},\mu_{2},\ldots\right)$ we will mean a tuple of
integers such that $\mu_{i}\ge\mu_{i+1}$. We write $\mu\le\lambda$
if $\mu_{i}\le\lambda_{i}$ for all $i$. Given a partition $\mu$
we denote by $\mu'$ the transpose partition.

For positive integers $k\le n$ we define 
\begin{align*}
\left[\begin{array}{c}
n\\
k
\end{array}\right]_{l} & =\frac{\prod_{i=1}^{n}\left(l^{i}-1\right)}{\prod_{i=1}^{k}\left(l^{i}-1\right)\prod_{i=1}^{n-k}\left(l^{i}-1\right)}
\end{align*}
 which is equal to the number of $k$-dimensional subspaces of $\mathbb{F}_{l}^{n}$.
\begin{prop}
Suppose $R$ is local and $R_{0}$ is a quotient of a Euclidean domain.
The $k$th moment of the function $M\mapsto l^{\mathrm{rk}_{l}M}$
on $S_{R}$ is
\begin{align*}
\sum_{M\in S_{R}}\mu_{R}\left(M\right)l^{k\mathrm{rk}_{l}M} & =\sum_{\mu\le\lambda_{m}}\prod_{j=1}^{m}l^{\mu'_{j+1}\left(k-\mu_{j+1}'\right)}\left[\begin{array}{c}
k-\mu'_{j+1}\\
\mu_{j}'-\mu'_{j+1}
\end{array}\right]_{l}
\end{align*}
 where $\lambda_{m}=\left(m,\ldots,m\right)$ is a partition of rank
$k$.
\end{prop}

\begin{proof}
We write the $k$th moment as
\[
\sum_{M\in S_{R}}\mu_{R}\left(M\right)l^{k\mathrm{rk}_{l}M}=\sum_{M\in S_{R}}\mu_{R}\left(M\right)\left|\mathrm{Hom}_{\mathbb{F}_{l}}\left(M,\mathbb{F}_{l}\right)\right|^{k}.
\]
Let $R_{0}=R/l$ and denote by $R_{0}^{\wedge}=\mathrm{Hom}_{\mathbb{F}_{l}}\left(R_{0},\mathbb{F}_{l}\right)$
the $\mathbb{F}_{l}$-dual. By the tensor-hom adjunction we have 
\begin{align*}
\mathrm{Hom}_{\mathbb{F}_{l}}\left(M,\mathbb{F}_{l}\right) & =\mathrm{Hom}_{\mathbb{F}_{l}}\left(M\otimes_{\mathbb{F}_{l}}R_{0},\mathbb{F}_{l}\right)\\
 & \cong\mathrm{Hom}_{R}\left(M,\mathrm{Hom}_{\mathbb{F}_{l}}\left(R_{0},\mathbb{F}_{l}\right)\right).
\end{align*}
 The action of $R$ on $R_{0}^{\wedge}$ can be described as follows.
Choose an $\mathbb{F}_{l}$-basis for $R_{0}$. This gives a map $R\longrightarrow\mathrm{Mat}_{\mathrm{rk}M}\left(\mathbb{F}_{l}\right)$.
Then for $f\in R_{0}^{\wedge}$ and $r\in R$ we have $r\cdot f\left(x\right)=f\left(r^{t}x\right)$
where $r^{t}$ is the transpose matrix. Then it is easy to see that
with the above choice of basis the canonical map $R_{0}\longrightarrow R_{0}^{\wedge}$
is an $R$-module isomorphism.

Thus we have 
\begin{align*}
\sum_{M\in S_{R}}\mu_{R}\left(M\right)\left|\mathrm{Hom}_{\mathbb{F}_{l}}\left(M,\mathbb{F}_{l}^{k}\right)\right| & =\sum_{M\in S_{R}}\mu_{R}\left(M\right)\left|\mathrm{Hom}_{R}\left(M,R_{0}^{k}\right)\right|.
\end{align*}
 We rewrite this in terms of surjection moments which we can apply
our previous results to. By rearranging the sum and applying Lemma
\ref{lem:measure moments are 1} we obtain 
\begin{align*}
\sum_{M\in S_{R}}\mu_{R}\left(M\right)\left|\mathrm{Hom}_{R_{0}}\left(M/lM,R_{0}^{k}\right)\right| & =\sum_{M\in S_{R}}\mu_{R}\left(M\right)\sum_{N\le R_{0}^{k}}\left|\mathrm{\mathrm{Surj}}_{R}\left(M,N\right)\right|\\
 & =\sum_{N\le R_{0}^{k}}\sum_{M\in S_{R}}\mu_{R}\left(M\right)\left|\mathrm{\mathrm{Surj}}_{R}\left(M,N\right)\right|\\
 & =\sum_{N\le R_{0}^{k}}1.
\end{align*}
 Recall $R_{0}$ is a quotient of a PID and we can assume $\pi^{m}R_{0}$
for some $m$, where $\left(\pi\right)$ is the maximal ideal of $R_{0}$.
Thus isomorphism classes of $R_{0}$-modules are in bijection with
partitions $\mu$ with $\mu_{i}\le m$ for all $i$.

Let $\lambda_{m}=\left(m,\ldots,m\right)$. Hence each submodule of
$R_{0}^{k}$ has a corresponding partition $\mu\subset\lambda_{m}$.
Given a partition $\mu$ the number of submodules of $R_{0}^{k}$
corresponding to $\mu$ is given by the following formula which can
be found for instance in \cite{numberofsubgroups,garrettnumberofsubgroups}
\begin{align*}
\prod_{j=1}^{m}l^{\mu'_{j+1}\left(\lambda_{j}'-\mu_{j+1}'\right)}\left[\begin{array}{c}
\lambda_{j}'-\mu'_{j+1}\\
\mu_{j}'-\mu_{j+1}
\end{array}\right] & =\prod_{j=1}^{m}l^{\mu'_{j+1}\left(k-\mu_{j+1}'\right)}\left[\begin{array}{c}
k-\mu'_{j+1}\\
\mu_{j}'-\mu'_{j+1}
\end{array}\right]_{l}
\end{align*}
 Thus we conclude
\[
\sum_{N\le R_{0}^{k}}1=\sum_{\mu\le\lambda_{m}}\prod_{j=1}^{m}l^{\mu'_{j+1}\left(k-\mu_{j+1}'\right)}\left[\begin{array}{c}
k-\mu'_{j+1}\\
\mu_{j}'-\mu_{j+1}
\end{array}\right]_{l}.
\]
\end{proof}

\section{\label{sec:Jacobians-of-hyperelliptic}Jacobians of hyperelliptic
curves}

Let $\mathcal{M}_{g}\left(\mathbb{F}_{q}\right)$ denote the moduli
space of hyperelliptic curves of genus $g$ defined over $\mathbb{F}_{q}$.
For each $C\in\mathcal{M}_{g}$ let $\mathrm{Jac}\left(C\right)$
denote the Jacobian. Let $F$ be the Frobenius operator.

Let $P\left(X\right)\in\mathbb{Z}_{l}\left[X\right]$ such that $P\left(0\right)$
is invertible in $\mathbb{Z}_{l}$. Define the ring $R=\mathbb{Z}_{l}\left[X\right]/\left\langle P\left(X\right)\right\rangle $.

Let $T_{l}\left(\mathrm{Jac}\left(C\right)\right)$ denote the $l$-adic
Tate module of $\mathrm{Jac}\left(C\right)$. We will view this as
a module over $\mathbb{Z}_{l}\left[X\right]$ with $X$ acting as
$F$. Let $G$ be a fixed finite etale group scheme. Lipnowski and
the second author proved a theorem \cite[Theorem 1.1]{lipnowskitsimerman}
about the averages of $\left|\mathrm{Surj}\left(\mathrm{Jac}\left(C\right),G\right)\right|$
as $C$ varies over $\mathcal{M}_{g}\left(\mathbb{F}_{q}\right)$
which roughly says that the average of this function approaches 1
as $g$ and $q$ go to infinity. We can combine this with the results
of Section \ref{subsec:Moments-and-limits} to prove the distributions
of the cokernels $T_{l}\left(\mathrm{Jac}\left(C\right)\right)/\left\langle P\left(F\right)\right\rangle $
approach the measure $\mu_{R}$ defined in Section \ref{sec:The-Cohen-Lenstra-measure}.

Call $\lambda\in\overline{\mathbb{Q}_{l}}$ a Weil $q$-number if
it satisfies $\left|\phi\left(\lambda\right)\right|=\sqrt{q}$ for
any embedding $\phi:\overline{\mathbb{Q}_{l}}\longrightarrow\mathbb{C}$.
All the eigenvalues of $F$ acting on $T_{l}\left(\mathrm{Jac}\left(C\right)\right)\otimes\overline{\mathbb{Q}_{l}}$
are Weil $q$-numbers \cite{delignWeilconjectures}. Though we prove
the next theorem under the assumption that the roots of $P\left(F\right)$
are not Weil-$q$ numbers, we expect it is true when restricted to
a subset of $\mathcal{M}_{g}$ where $T_{l}\left(\mathrm{Jac}\left(C\right)\right)/\left\langle P\left(F\right)\right\rangle $
is finite.

Fix $g,q$ and define the probability measure on $S_{R}$ by 
\[
\mu_{g,q}\left(M\right)=\frac{\left|\left\{ C\in\mathcal{M}_{g}\left(\mathbb{F}_{q}\right)\mid T_{l}\left(\mathrm{Jac}\left(C\right)\right)/\left\langle P\left(F\right)\right\rangle \cong M\right\} \right|}{\left|\mathcal{M}_{g}\left(\mathbb{F}_{q}\right)\right|}.
\]

\begin{thm}
\label{thm:cokernel density}Assume $\prod_{i=1}^{s}\eta\left(\mathbb{F}_{i}\right)>1/2$.
Suppose the roots of $P\left(F\right)$ are not Weil $q$-numbers.
Then the sequence of probability measures $\left\{ \mu_{g,q}\right\} $
converges weakly to $\mu_{R}$ along any sequence with $l\nmid P\left(q\right)$.
\end{thm}

\begin{proof}
Fix $G\in S_{R}$.  The proof of Theorem 1.1 in \cite{lipnowskitsimerman}
shows that there exists $N$ such that if $g,q>N$ and $l\nmid P\left(q\right)$
then $\left|\int_{S_{R}}\left|\mathrm{Surj}\left(M,G\right)\right|d\mu_{g,q}-1\right|<\epsilon$.
Hence the by Theorem \ref{thm:conv of moments gives measure} $\lim_{g,q\longrightarrow\infty}\mu_{g,q}=\nu$
for some probability measure $\nu$ satisfying $\int_{S_{R}}\left|\mathrm{Surj}\left(M,G\right)\right|d\nu=1$
for all $G\in S_{R}$. By Lemma \ref{lem:mu_R unique} $\nu=\mu_{R}$.
\end{proof}
\begin{rem}
The version of Theorem \ref{thm:cokernel density} for abelian groups
proved in \cite{EVW} is stated using iterated limits, by first taking
the limsup and liminf in $g$ and then taking the limit in $q$. That
version is implied by the one above.
\end{rem}

Theorem \ref{thm:cokernel density} proves a more general form of
Conjecture 1.1 from \cite{lipnowskitsimerman} which applies to $\mathrm{Jac}\left(C\right)\left[P\left(F\right)\right]$
for any $P\left(X\right)\in\mathbb{Z}_{l}\left[X\right]$ in the case
$l\nmid P\left(q\right)$ (which implies $\land^{2}G=0$ for all $\mathbb{Z}_{l}\left[F\right]/\left\langle P\left(F\right)\right\rangle $-modules
$G$), up to the limit in $q$.

For the applications in the next section we need a different interpretation
of the measures $\mu_{g,q}$. Define
\[
\mu_{g,q}^{\prime}\left(M\right)=\frac{\left|\left\{ C\in\mathcal{M}_{g}\left(\mathbb{F}_{q}\right)\mid\mathrm{Jac}\left(C\right)\left[P\left(F\right)\right]\cong M\right\} \right|}{\left|\mathcal{M}_{g}\left(\mathbb{F}_{q}\right)\right|}.
\]

We have the following lemma relating the measures $\mu'_{g,q}$ and
$\mu_{g,q}$. Though we prove it under the assumption that the roots
of $P\left(F\right)$ are not Weil-$q$ numbers, we expect it is true
in general.
\begin{lem}
\label{lem:mu' equals mu}Suppose the roots of $P\left(X\right)\in\mathbb{Z}_{l}\left[X\right]$
are not Weil $q$-numbers. Then for any $C\in\mathcal{M}_{g}\left(\mathbb{F}_{q}\right)$
\[
\mathrm{Jac}\left(C\right)\left[P\left(F\right)\right]\cong T_{l}\left(\mathrm{Jac}\left(C\right)\right)/\left\langle P\left(F\right)\right\rangle .
\]
In particular, for all $g,q$ we have $\mu_{g,q}^{\prime}=\mu_{g,q}$.
\end{lem}

\begin{proof}
Consider the exact sequence 
\[
0\longrightarrow T_{l}\left(\mathrm{Jac}\left(C\right)\right)\longrightarrow T_{l}\left(\mathrm{Jac}\left(C\right)\right)\longrightarrow\mathrm{Jac}\left(C\right)\left[l^{n}\right]\longrightarrow0
\]
where the first map is multiplication by $l^{n}$. Apply the snake
lemma to the diagram consisting of two rows being the above sequence
and the vertical maps being multiplication by $P\left(F\right)$.

We claim $\ker P\left(F\right)=0$ for $P\left(F\right)$ acting on
$T_{l}\left(\mathrm{Jac}\left(C\right)\right)$. Suppose $\ker P\left(F\right)\neq0$.
Then $\left(F-\lambda\right)x=0$ for some $x\in T_{l}\left(\mathrm{Jac}\left(C\right)\right)\otimes\overline{\mathbb{Q}_{l}}$
and $\lambda\in\overline{\mathbb{Q}_{l}}$ some root of $P\left(F\right)$.
This implies $\lambda$ is an eigenvalue of $F$, contradicting that
$\lambda$ is not a Weil $q$-number. Thus the connecting homomorphism
is injective.

We can let $n$ be large enough such that $l^{n}T_{l}\left(\mathrm{Jac}\left(C\right)\right)\subset P\left(F\right)T_{l}\left(\mathrm{Jac}\left(C\right)\right)$,
making the connecting homomorphism surjective. We can also make $n$
large enough such that $\mathrm{Jac}\left(C\right)\left[l^{n}\right]\left[P\left(F\right)\right]=\mathrm{Jac}\left(C\right)\left[P\left(F\right)\right]$

Thus the connecting homomorphism induces an isomorphism between $\mathrm{Jac}\left(C\right)\left[l^{n}\right]\left[P\left(F\right)\right]$
and $T_{l}\left(\mathrm{Jac}\left(C\right)\right)/\left\langle P\left(F\right)\right\rangle $.
\end{proof}

\section{\label{sec:Densities-of-eigenspaces}Densities of eigenspaces of
frobenius}

As in Section \ref{sec:Jacobians-of-hyperelliptic} let $R=\mathbb{Z}_{l}\left[X\right]/\left\langle P\left(X\right)\right\rangle $
and view $\mathrm{Jac}\left(C\right)\left[P\left(F\right)\right]$
as an $R$-module with $X$ acting as $F$.

\subsection{Independence of torsion subgroups of $\mathrm{Jac}\left(C\right)$}
\begin{lem}
\label{lem:factor mu}Let $P\left(X\right)=\prod_{i=1}^{s}P_{i}\left(X\right)$
where $P_{i}\left(X\right)\in\mathbb{Z}_{l}\left[X\right]$ are irreducible.
Suppose the $P_{i}\left(X\right)$ are coprime modulo $l$. Let $R_{i}=\mathbb{Z}_{l}\left[X\right]/\left(P_{i}\left(X\right)\right)$.
For any $M\in S_{R}$ we have 
\[
\mu_{R}\left(M\right)=\prod_{i=1}^{s}\mu_{R_{i}}\left(M\otimes_{R}R_{i}\right).
\]
\end{lem}

\begin{proof}
Note $R\cong\prod_{i=1}^{s}R_{i}$ since the $P_{i}\left(X\right)$
are coprime modulo $l$. Let $M_{i}=M\otimes_{R}R_{i}$. Then $M\cong\oplus_{i=1}^{s}M_{i}$.
Let $m_{i}=\left\langle l,P_{i}\left(X\right)\right\rangle \subset\mathbb{Z}_{l}\left[X\right]$.
Then $m_{i}R_{i}$ is the unique maximal ideal of $R_{i}$. Since
the $P_{i}\left(X\right)$ are not coprime modulo $l$ we see that
$P_{i}\left(X\right)\notin m_{j}$ for $i\neq j$. This implies $\mathrm{Hom}_{R}\left(M_{i},M_{j}\right)=0$
for $i\neq j$.

Hence $\mathrm{End}_{R}\left(M\right)=\prod_{i=1}^{s}\mathrm{End}\left(M_{i}\right)$
and so $\mathrm{Aut}_{R}\left(M\right)=\prod_{i=1}^{s}\mathrm{Aut}\left(M_{i}\right)$.
The maximal ideals of $R$ are $m_{i}R$ so $\mathbb{F}_{i}=\mathbb{F}_{l}$.
Recall the constants $c_{R_{i}},c_{R}$ given by Theorem \ref{thm:The-measure-}
and Lemma \ref{lem:ring condition} (since $\mathbb{Z}_{l}$ is contained
in $R_{i}$ and $R$). We have 
\begin{align*}
c_{R} & =\lim_{n\longrightarrow\infty}\prod_{i=1}^{s}\left|\mathrm{GL}_{n}\left(\mathbb{F}_{l}\right)\right|/\left|M_{n}\left(\mathbb{F}_{l}\right)\right|\\
 & =\prod_{i=1}^{s}c_{R_{i}}
\end{align*}
 and the result follows.
\end{proof}
We have the following consequence of Theorem \ref{thm:cokernel density}
which says that any finitely many given eigenspaces of $\mathrm{Frob}$
are distributed independently in the $\left(g,q\right)$ limit.
\begin{cor}
\label{cor:indep eigenspaces}Let $\epsilon>0$. Let $P_{i}\left(X\right)\in\mathbb{Z}_{l}\left[X\right]$
be irreducible, pairwise coprime modulo $l$ and let $n_{i}\in\mathbb{Z}$
for $i=1,\ldots,s$. Suppose the roots of $P\left(X\right)$ are not
Weil $q$-numbers and assume $\prod_{i=1}^{s}\eta\left(\mathbb{F}_{i}\right)>1/2$.
Suppose $M\in S_{R}$ and $M\cong\oplus_{i=1}^{s}M_{i}$. There exists
$N$ such that if $g,q>N$ then 
\[
\left|\mathrm{Prob}\left(\mathrm{Jac}\left(C\right)\left[\prod_{i=1}^{s}P_{i}\left(F\right)^{n_{i}}\right]\cong M\right)-\prod_{i=1}^{s}\mathrm{Prob}\left(\mathrm{Jac}\left(C\right)\left[P_{i}\left(F\right)^{n_{i}}\right]\cong M_{i}\right)\right|<\epsilon.
\]
\end{cor}

\begin{proof}
Follows immediately from Theorem \ref{thm:cokernel density} and Lemma
\ref{lem:factor mu}.
\end{proof}

\subsection{Densities of divisors of the characteristic polynomial of Frobenius}

We can apply Theorem \ref{thm:cokernel density} to obtain densities
for the characteristic polynomial of the Frobenius acting on the Tate
module of the Jacobian of hyperelliptic curves.

Let $C\in\mathcal{M}_{g}\left(\mathbb{F}_{q}\right)$. Let $P_{C}\left(X\right)$
be the characteristic polynomial of $F$ acting on $\mathrm{Jac}\left(C\right)\left[l\right]$.

We can use the measure $\mu_{R}$ to detect how often $P_{i}\left(X\right)^{m}\mid\mid P_{C}\left(X\right)$.
\begin{prop}
\label{prop:char poly prob}Let $P\left(X\right)=\prod_{i=1}^{s}P_{i}\left(X\right)^{m_{i}}$
where each $P_{i}\left(X\right)\in\mathbb{F}_{l}\left[X\right]$ is
an irreducible polynomial. Let $R_{i}=\mathbb{F}_{l}\left[X\right]/\left(P_{i}\left(X\right)^{m_{i}+1}\right)$
and $\mathbb{F}_{i}=R_{i}/\left(P_{i}\left(X\right)\right)$. Assume
$\prod_{i=1}^{s}\eta\left(\mathbb{F}_{i}\right)>1/2$. There exists
$N$ such that if $g,q>N$
\[
\left|\mathrm{Prob}\left(P\left(X\right)\mid\mid P_{C}\left(X\right)\right)-\prod_{i=1}^{s}\sum_{\substack{M\in S_{R_{i}}\\
\dim_{\mathbb{F}_{l}}M=m_{i}
}
}\mu_{R_{i}}\left(M\right)\right|<\epsilon.
\]
\end{prop}

\begin{proof}
For any curve $C\in\mathcal{M}_{g}$ $\mathrm{Jac}\left(C\right)\left[l\right]$
is an $\mathbb{F}_{l}\left[X\right]$-module. Since $\mathbb{F}_{l}\left[X\right]$
is a PID we can apply the structure theorem. For any $M\in S_{\mathbb{F}_{l}\left[X\right]}$
and any irreducible $Q\left(X\right)\in\mathbb{F}_{l}\left[X\right]$
let $E_{Q}\left(M\right)=M\left[Q\left(X\right)^{\infty}\right]$.
Let $\mathbb{F}=\mathbb{F}_{l}\left[X\right]/\left(Q\left(X\right)\right)$
be the residue field. Then $\dim_{\mathbb{F}_{l}}E_{Q}\left(M\right)$
is the exact power of $Q\left(X\right)$ dividing the characteristic
polynomial of $F$ acting on $M$ and it is easy to see that $\dim_{\mathbb{F}_{l}}E_{Q}\left(M\right)=m$
if and only if $\dim_{\mathbb{F}_{l}}\ker\left(Q\left(X\right)^{m+1}\right)=m$
(note $m$ has to be divisible by $\left[\mathbb{F}:\mathbb{F}_{l}\right]$).

For each $i$ let $\widehat{P_{i}}\left(X\right)$ be a lift of $P_{i}\left(X\right)$
to $\mathbb{Z}_{l}\left[X\right]$ chosen such that the roots of $P_{i}'\left(X\right)$
are not Weil $q$-numbers. Let $R=\mathbb{Z}_{l}\left[X\right]/\left(\prod_{i=1}^{s}\widehat{P_{i}}\left(X\right)^{m_{i}+1}\right)$.
For any $M\in S_{R}$ we denote $M_{i}=M\otimes_{R}R_{i}$. Thus
\begin{align*}
 & \mathrm{Prob}\left(\dim_{\mathbb{F}_{l}}E_{P_{i}}\left(\mathrm{Jac}\left(C\right)\left[l\right]\right)=m_{i}\ \forall i\right)\\
 & =\sum_{\substack{M\in S_{R}\\
\dim_{\mathbb{F}_{l}}M\left[l\right]\left[\widehat{P_{i}}\left(X\right)^{m_{i}+1}\right]=m_{i}\ \forall i
}
}\mathrm{Prob}\left(\mathrm{Jac}\left(C\right)\left[\prod_{i=1}^{s}\widehat{P_{i}}\left(X\right)^{m_{i}+1}\right]\cong M\right)\\
 & =\sum_{\substack{M\in S_{R}\\
\dim_{\mathbb{F}_{l}}M_{i}=m_{i}\ \forall i
}
}\mu_{R,g,q}^{\prime}\left(M\right)
\end{align*}
where in the last line we change the condition in the sum using $\dim_{\mathbb{F}_{l}}M\left[l\right]\left[\widehat{P_{i}}\left(X\right)^{m_{i}+1}\right]=\dim_{\mathbb{F}_{l}}M_{i}$.
This follows from considering the exact sequences obtained from multiplication
by $l$ and by $\widehat{P_{i}}\left(X\right)^{m_{i}+1}$ using that
$M$ is finite to get
\begin{align*}
\left|M\left[l\right]\left[\widehat{P_{i}}\left(X\right)^{m_{i}+1}\right]\right| & =\left|M\left[\widehat{P_{i}}\left(X\right)^{m_{i}+1}\right]/\left(l\right)\right|\\
 & =\left|M/\left(\widehat{P_{i}}\left(X\right)^{m_{i}+1},l\right)\right|.
\end{align*}

Let $T_{m_{i}}=\left\{ M\in S_{R}\mid\dim_{\mathbb{F}_{l}}M_{i}=m_{i}\right\} $.
Then by Theorem \ref{thm:cokernel density} for any $\epsilon>0$
there exists $N$ such that if $g,q>N$ then $\left|\int_{\cap_{i}T_{m_{i}}}d\mu_{R,g,q}^{\prime}-\int_{\cap_{i}T_{m_{i}}}d\mu_{R}\right|<\epsilon.$

Let $\widehat{R_{i}}=\mathbb{Z}_{l}\left[X\right]/\left(P_{i}\left(X\right)^{m_{i}+1}\right)$.
Then it follows from Lemma \ref{lem:factor mu} that

\begin{align*}
\sum_{\substack{M\in S_{R}\\
\dim_{\mathbb{F}_{l}}M_{i}=m_{i}
}
}\mu_{R}\left(M\right) & =\sum_{\substack{M\in S_{R}\\
\dim_{\mathbb{F}_{l}}M_{i}=m_{i}
}
}\prod_{i=1}^{s}\mu_{\widehat{R_{i}}}\left(M\otimes_{R}R_{i}\right)\\
 & =\prod_{i=1}^{s}\sum_{\substack{M\in S_{\widehat{R_{i}}}\\
\dim_{\mathbb{F}_{l}}M/lM=m_{i}
}
}\mu_{\widehat{R_{i}}}\left(M\right).
\end{align*}
Then the result follows by Lemma \ref{lem:prob reduced mod p}.
\end{proof}
\begin{rem}
\label{rem:consequences for char poly}This lemma shows that if we
fix any $P_{i}\left(X\right)\in\mathbb{Z}_{l}\left[X\right]$ for
$i=1,\ldots,s$ irreducible and let $\nu$ be the measure on $\mathbb{Z}_{\ge0}^{s}$
given by 
\[
\nu\left(m_{1},\ldots m_{s}\right)=\sum_{\substack{M\in S_{R}\\
\dim_{\mathbb{F}_{i}}M_{i}=m_{i}
}
}\mu_{R}\left(M\right)
\]
 then $\mathrm{Prob}\left(\prod_{i=1}^{s}P_{i}\left(X\right)^{m_{i}}\mid\mid P_{C}\left(X\right)\right)$
approaches $\nu$ in the $g,q$ limit. Note the $m_{i}$ are allowed
to be 0, that is we can exclude any finite set of divisors from $P_{C}\left(X\right)$.

It is clear that $\mathrm{Prob}\left(P\left(X\right)=P_{C}\left(X\right)\right)=0$
for any fixed $P\left(X\right)\in\mathbb{Z}_{l}\left[X\right]$ since
the characteristic polynomial determines the $\mathbb{F}_{l}\left[X\right]$-module
$\mathrm{Jac}C\left[l\right]$ up to a finite set of possibilities,
and the density of each of these is 0.

A question which is not answered by the above is: given any infinite
set of polynomials $\left\{ P_{i}\left(X\right)\right\} _{i\in\mathcal{I}}$
and $m_{i}\ge0$, what is $\mathrm{Prob}\left(\cup_{i\in\mathcal{I}}\left\{ P_{i}\left(X\right)^{m_{i}}\mid\mid P_{C}\left(X\right)\right\} \right)$?\\
\end{rem}

We now focus on explicitly computing each factor in Proposition \ref{prop:char poly prob}.
Until further notice let $R=\mathbb{F}_{l}\left[X\right]/\left(P\left(X\right)^{m+1}\right)$
for $P\left(X\right)\in\mathbb{F}_{l}\left[X\right]$ irreducible.
For $M\in S_{R}$ let $\mathrm{rk}_{l}M=\dim_{\mathbb{F}_{l}}M$.
We want to compute 
\[
\sum_{\substack{M\in S_{R}\\
\mathrm{rk}_{l}M=m
}
}\mu_{R}\left(M\right)=\mu_{R}\left(\left\{ M\in S_{R}\mid\mathrm{rk}_{l}M=m\right\} \right)
\]
 for all $m$ (note $R$ also depends on $m$).
\begin{lem}
\label{lem:rank prob polynomials}We have 
\[
\mu_{R}\left(\left\{ M\in S_{R}\mid\mathrm{rk}_{l}M=m\right\} \right)=\eta\left(\mathbb{F}\right)\sum_{\substack{M\in S_{R}\\
\mathrm{rk}_{l}M=m
}
}\frac{1}{\left|\mathrm{Aut}_{R}\left(M\right)\right|}
\]
\end{lem}

\begin{proof}
Let $S_{R,m}=\left\{ M\in S_{R}\mid\dim_{\mathbb{F}_{l}}M=m\right\} $.
Applying the formula from Theorem \ref{thm:The-measure-} we have
\[
\mu_{R}\left(\left\{ M\in S_{R}\mid\mathrm{rk}_{l}M=m\right\} \right)=\sum_{j=0}^{m'/k}\sum_{M\in S_{R,m}\cap T_{R,j}}\frac{c_{R,j}}{\left|\mathrm{Aut}_{R}\left(M\right)\right|}.
\]
 Recall we denote by $T_{R}$ the support of $\mu_{R}$ which is a
disjoint union $T_{R}=\bigcup_{j=0}^{\infty}T_{R,j}$ (see (\ref{eq:partition TR})).

We claim that $S_{R,m}\subset T_{R,0}$. Let $\pi=P\left(X\right)$
so that $\left(\pi\right)$ is the unique maximal ideal of $R$. From
the proof of Lemma \ref{lem:characterize T_R} we see that if
\[
0\longrightarrow K\longrightarrow R^{n}\longrightarrow M\longrightarrow0
\]
 then $d_{m}=n-\dim_{\mathbb{F}}K\otimes_{R}\mathbb{F}.$ Hence we
must show that $\dim_{\mathbb{F}}K\otimes_{R}\mathbb{F}=n$.

Note $R\left[\pi\right]=\pi^{m}R$. Let $\phi:R^{n}\longrightarrow M$
be the map above, so $K=\ker\phi$. Since $\dim_{\mathbb{F}_{l}}M=m$
we have $\left|M\right|=l^{m}$. Every $R$-module is of the form
\[
\bigoplus_{i=1}^{n}R/\left(\pi\right)^{n_{i}}
\]
 which combined with the previous fact implies that $\pi^{m}M=0$.
Thus $R^{n}\left[\pi\right]\subseteq K\left[\pi\right]$ so we have
equality $R^{n}\left[\pi\right]=K\left[\pi\right]$.

Finally since $R$ and $K$ are finite we have $\left|K\otimes_{R}\mathbb{F}\right|=\left|K\left[\pi\right]\right|$
and $\left|R\otimes_{R}\mathbb{F}\right|=\left|R\left[\pi\right]\right|$.
We conclude that $\left|K\otimes_{R}\mathbb{F}\right|=\left|R^{n}\otimes_{R}\mathbb{F}\right|=\left|\mathbb{F}\right|^{n}$.
This proves the claim.

By Theorem \ref{thm:The-measure-} applied to $R$ we have that for
each $M\in S_{R,m}$ 
\[
\mu_{R}\left(M\right)=\frac{\eta\left(\mathbb{F}\right)}{\left|\mathrm{Aut}_{R}\left(M\right)\right|}.
\]
 This completes the proof.
\end{proof}
It is also possible to apply Corollary \ref{cor:explicit formula}
to the result of Lemma \ref{lem:rank prob polynomials} to get a more
explicit answer.

To demonstrate we compute some simple examples.
\begin{cor}
\label{cor:examples}For $l>2$ and any $a\in\mathbb{F}_{l}$ we have
\[
\mathrm{Prob}\left(\left(X-a\right)\mid\mid P_{C}\left(X\right)\right)\longrightarrow\eta\left(\mathbb{F}_{l}\right)\left[\frac{1}{\left(l^{2}-1\right)\left(l^{2}-l\right)}+\frac{1}{l^{2}-l}\right]
\]
 and 
\[
\mathrm{Prob}\left(\left(X-a\right)\nmid P_{C}\left(X\right)\right)\longrightarrow\eta\left(\mathbb{F}_{l}\right)
\]
 in the $g,q$-limit.
\end{cor}

\begin{proof}
We apply Lemma \ref{lem:rank prob polynomials}. For the first part
we have $R=\mathbb{F}_{l}\left[X\right]/\left(X-a\right)^{2}$ and
$\mathbb{F}=\mathbb{F}_{l}$. The only two $R$-modules with $\mathbb{F}_{l}$-rank
$2$ are 
\[
R/\left(X-a\right)\oplus R/\left(X-a\right),\ R
\]
 and applying the formula from Corollary \ref{cor:explicit formula}
gives the desired result.

For the second part $R=\mathbb{F}_{l}\left[X\right]/\left(X-a\right)$,
$\mathbb{F}=\mathbb{F}_{l}$ and the only module in the sum is the
trivial one.
\end{proof}
\bibliographystyle{alpha}
\bibliography{modules}

$\ $

\textsc{\small{}Jack Klys}{\small\par}

\textsc{\small{}Department of Mathematics and Statistics, University
of Calgary, Canada}{\small\par}

{\small{}$\mathtt{jack.klys@ucalgary.ca}$}\\

\textsc{\small{}Jacob Tsimerman}{\small\par}

\textsc{\small{}Department of Mathematics, University of Toronto,
Canada}{\small\par}

{\small{}$\mathtt{jacobt@math.toronto.edu}$}{\small\par}
\end{document}